\title{Algorithms for Chip-Firing on Weighted Graphs}
\author{Ben Doyle}

\documentclass{article}
\usepackage[T1]{fontenc}
\usepackage[utf8]{inputenc}
\usepackage{tikz}
\usetikzlibrary{shapes,fit,cd,arrows.meta}
\usepackage{amsmath,amsthm,amssymb}
\usepackage[ruled,vlined]{algorithm2e}
\usepackage{comment}
\usepackage[normalem]{ulem}
\usepackage{xcolor}
\usepackage{soul}
\usepackage{enumitem}
\usepackage{multirow,booktabs}
\usepackage[justification=centering]{caption}
\usepackage[nodisplayskipstretch]{setspace}
\usepackage{geometry}
\newcommand{\Div}{\text{Div}}

\newcommand{\Prin}{\text{Prin}}
\newcommand{\lcm}{\mathrm{lcm}}
\newcommand{\Jac}{\text{Jac}}
\newcommand{\Pic}{\text{Pic}}
\newcommand{\Z}{\mathbb{Z}}
\newcommand{\N}{\mathbb{N}}

\newcommand{\Ind}{\text{Ind}}
\newcommand{\val}{\text{val}}

\newtheorem{theorem}{Theorem}[section]
\theoremstyle{definition}
\newtheorem{definition}[theorem]{Definition}
\newtheorem{lemma}[theorem]{Lemma}
\newtheorem{corollary}[theorem]{Corollary}

\newtheorem{proposition}[theorem]{Proposition}

\begin{document}

\begin{center}
    \vspace{0.25in}
    
    {\large\uppercase{\textbf{Algorithms for Chip-Firing on Weighted Graphs}}}
    
    \vspace{0.25in}
    
    \uppercase{Ben Doyle}
\end{center}

\begin{abstract}
    \noindent {\small\textsc{Abstract.}} We extend the notion of chip-firing to weighted graphs, and generalize the Greedy Algorithm and Dhar's Burning Algorithm to weighted graphs. For a vertex $q \in V(\Gamma)$, we give an upper bound for the number of linearly equivalent $q$-reduced divisors. Finally, we illustrate a method of finding all maximal unwinnable divisors on weighted graphs.
\end{abstract}

\section{Introduction}

It has long been recognized that finite graphs are discrete analogues of Riemann surfaces, and in particular, divisor theory on graphs is quite similar to divisor theory on Riemann surfaces (this is not simply a coincidence, as detailed by Baker \cite{baker}). A divisor on a graph can be thought of as a (discrete) assignment of wealth and/or debt to each vertex, and chip-firing serves as a natural method of navigating between different divisors on graphs. If one divisor can be reached from another by some sequence of chip-firing moves (known as a ``firing script''), then we say that these divisors are linearly equivalent. A fundamental task associated to chip-firing is known as the Dollar Game: given a particular divisor, is there a linearly equivalent divisor with the property that every vertex is out of debt?

The Greedy Algorithm and Dhar's Burning Algorithm are two methods for determining the winnability of a divisor with regards to the Dollar Game on a graph. These algorithms, and much more theory regarding chip-firing, are explored in Corry and Perkinson's text, \emph{Divisors and Sandpiles: An Introduction to Chip-Firing} \cite{perkinson}, culminating in a proof of the Riemann-Roch Theorem for Graphs:

\begin{theorem}[Riemann-Roch]
Let $D$ be a divisor on a graph $\Gamma$ of genus $g = |E|-|V|+1$ with canonical divisor $K$. Then
\[r(D)-r(K-D) = \deg(D)+1-g,\]
where $r(D)$ is the \emph{rank} of $D$, that is, $r(D)$ is the largest $k$ such that $D-D'$ is winnable for all $D' \in \Div^k(\Gamma)$.
\label{Riemann-Roch}
\end{theorem}

The proof of the Riemann-Roch Theorem detailed in \cite{perkinson} (originally proven by Baker and Norine in \cite{baker2}) involves the development of $q$-reduced, maximal unwinnable, and canonical divisors using Dhar's Burning Algorithm.

In this paper we construct generalizations of the Greedy Algorithm and Dhar's Burning Algorithm to weighted graphs. From the Modified Burning Algorithm, we demonstrate that not all $q$-reduced divisors are unique on weighted graphs, and define $q$-classes as a consequence. Finally, we demonstrate a process by which we can find all $q$-reduced maximally unwinnable divisors on a given graph $\Gamma$.

\subsection*{Notation and Terminology}

A \emph{graph} $\Gamma$ is a set $V(\Gamma)$ of vertices, a set $H(\Gamma)$ of half-edges, a root map $r:H(\Gamma)\rightarrow V(\Gamma)$, and an involution $h \mapsto \overline{h}$ of $H(\Gamma)$. We refer to an orbit under the involution as an \emph{edge}, and denote the set of edges of $\Gamma$ by $E(\Gamma)$. For the purposes of this paper, a graph $\Gamma$ is assumed to be a finite, connected multigraph with no loops or legs. For vertices $u,v \in V(\Gamma)$, we denote by $E(v)$ as the set of all edges incident to $v$, and denote by $E(u,v)$ the set of all edges connecting $u$ and $v$. The \emph{valency} of a vertex is $\val(v) = |E(v)|$.

A \emph{weighted graph} is a graph along with weight maps $w_V:V(\Gamma) \rightarrow \Z^+$ and $w_E:E(\Gamma) \rightarrow \Z^+$ (we typically denote both maps simply by $w$) such that the weight of an edge divides the weights of each of its root vertices. There is a natural extension of valency to weighted graphs, which we refer to as the \emph{weighted valency}, defined by
\[\val(v) = \sum_{e \in E(v)} \dfrac{w(v)}{w(e)}.\]

Finally, let a group action of a group $G$ on a graph $\Gamma$ be a homomorphism $G \rightarrow \text{Aut}(\Gamma)$. Then a \emph{quotient graph} $\Gamma/G$, where $\Gamma$ is a graph and $G$ is a group acting on $\Gamma$, is defined by letting $V(\Gamma/G) = V(\Gamma)/G$, letting $H(\Gamma/G) = H(\Gamma)/G$, and preserving endpoint and involution maps, i.e. if $O(h)$ is the orbit of a half-edge $h$ under $G$, then
\[r(O(h)) = O(r(h)),\quad \overline{O(h)} = O(\overline{h}).\]
It is assumed that a group action does not permit any automorphisms which map a half-edge $h$ to its involution $\overline{h}$ or a vertex $v$ to any adjacent vertices, so as to avoid quotient graphs with legs or loops.

\section{Background}

We first define standard chip-firing on graphs. Define a \emph{divisor} on a graph $\Gamma$ to be a member of the free abelian group
\[\Div(\Gamma) = \Z V(\Gamma) = \left\{\sum_{v \in V} D(v)v:D(v) \in \Z\right\},\]
and the \emph{degree} of a divisor $D$ as
\[\deg(D) = \sum_{v \in V(\Gamma)} D(v).\]
(We also sometimes say that $D(v)$ is the \emph{degree} of $v$). A divisor $D$ is called \emph{effective} if $D(v) \geq 0$ for all $v \in V(\Gamma)$, and we say $D \geq 0$. In general, if $D_1(v) \geq D_2(v)$ for all $v \in V(\Gamma)$, we say $D_1 \geq D_2$. We denote the set of all divisors of degree $k$ on $\Gamma$ as $\Div^k(\Gamma)$. 

\begin{definition}
Suppose $\Gamma$ is an unweighted graph. Let $D,D' \in \Div(\Gamma)$ and $v \in V(\Gamma)$. Then $D'$ is obtained from $D$ by a \emph{lending move at v} if for $u \neq v$,
\[D'(u) = D(u) + |E(u,v)|, \quad \text{and} \quad D'(v) = D(v)-\val(v).\]
Similarly, $D'$ is obtained from $D$ by a \emph{borrowing move at v} if for $u \neq v$,
\[D'(u) = D(u) - |E(u,v)|, \quad \text{and} \quad D'(v) = D(v)+\val(v).\]
\end{definition}

Less technically, a lending move at $v$ is the action of sending a single chip from $v$ along each incident edge to each adjacent vertex, and a borrowing move at $v$ is the inverse of a lending move (See Figure \ref{borrow-lend}). 

\begin{figure}[ht]
    \centering
    \begin{tikzpicture}[node distance={15mm},main/.style = {draw,circle,fill, inner sep=2pt}]
    \begin{scope}
    \node[main,label = {[red]right:-2}] at (0:1cm) (1) {};
    \node[main,label = {[blue]above:1}] at (90:1cm) (2) {};
    \node[main,label = {[blue]left:3}] at (180:1cm) (3) {};
    \node[main,label = {[red]below:-1}] at (270:1cm) (4) {};
    
    \draw[-{Latex[length=3mm]},blue] (3) to (1);
    \draw[-{Latex[length=3mm]},blue] (3) to[out=75,in=195] (2);
    \draw[-{Latex[length=3mm]},blue] (3) to[out=15,in=255] (2);
    \draw[-{Latex[length=3mm]},blue] (3) to (4);
    \draw (1) to (2);
    \draw (1) to (4);
    \end{scope}
    
    \begin{scope}[xshift=5cm]
    \node[main,label = {[red]right:-1}] at (0:1cm) (1) {};
    \node[main,label = {[blue]above:3}] at (90:1cm) (2) {};
    \node[main,label = {[red]left:-1}] at (180:1cm) (3) {};
    \node[main,label = {[blue]below:0}] at (270:1cm) (4) {};
    
    \draw[-{Latex[length=3mm]},red] (3) to (1);
    \draw (3) to[out=75,in=195] (2);
    \draw (3) to[out=15,in=255] (2);
    \draw (3) to (4);
    \draw[-{Latex[length=3mm]},red] (2) to (1);
    \draw[-{Latex[length=3mm]},red] (4) to (1);
    \end{scope}
    
    \begin{scope}[xshift=10cm]
    \node[main,label = {[blue]right:2}] at (0:1cm) (1) {};
    \node[main,label = {[blue]above:2}] at (90:1cm) (2) {};
    \node[main,label = {[red]left:-2}] at (180:1cm) (3) {};
    \node[main,label = {[red]below:-1}] at (270:1cm) (4) {};
    
    \draw (3) to (1);
    \draw (3) to[out=75,in=195] (2);
    \draw (3) to[out=15,in=255] (2);
    \draw (3) to (4);
    \draw (2) to (1);
    \draw (4) to (1);
    \end{scope}
    \draw[->] (2,0) to (3,0);
    \draw[->] (7,0) to (8,0);
    \end{tikzpicture}
    \caption[Borrowing and lending on unweighted graphs]{The effects of a lending move at the left vertex followed by a borrowing move at the right vertex.}
    \label{borrow-lend}
\end{figure}
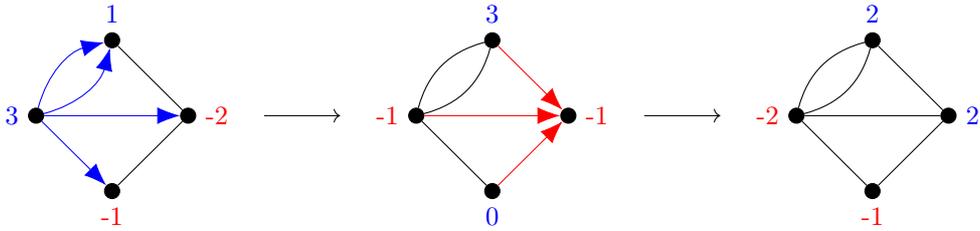

A collection of lending and borrowing moves is referred to as a \emph{firing script}, and is written as a function $\sigma:V(\Gamma)\rightarrow \Z$, where $\sigma(v)$ is the number of lending moves at $v$. The set of all firing scripts is denoted by $\mathcal{M}(\Gamma)$. Firing scripts provide a way to navigate between divisors in $\Div^k(\Gamma)$, and in fact form an equivalence relation over $\Div^k(\Gamma)$. We say that two divisors $D_1,D_2 \in \Div^k(\Gamma)$ are \emph{linearly equivalent} if there exists some firing script which takes $D_1$ to $D_2$. We denote by $[D]$ the set of all divisors linearly equivalent to $D$.
\begin{definition}
If $D \in \Div^0(\Gamma)$ is linearly equivalent to the zero divisor, then we say that $D$ is a \emph{principal divisor}. The principal divisors form a group, which we denote by $\Prin(\Gamma)$. We define the \emph{Picard group} of $\Gamma$ as $\Pic(\Gamma) = \Div(\Gamma)/\Prin(\Gamma)$ and the \emph{Jacobian group} of $\Gamma$ as $\Jac(\Gamma) = \Div^0(\Gamma)/\Prin(\Gamma)$.
\end{definition}

\begin{definition}
Let $\Gamma$ be a graph with $n$ vertices $v_1,...,v_n$. We define the \emph{Laplacian} of $\Gamma$ to be the $n \times n$ matrix $L$ with

\[L_{ij} = \begin{cases}
\val(v_i) \quad& \text{if } i = j \\
-|E(v_i,v_j)| \quad& \text{if } i \neq j.
\end{cases}\]
\end{definition}

Thus if $D$ is a divisor on $\Gamma$, and $\sigma = (a_1,a_2,...,a_n)$ is a firing script on $\Gamma$ (i.e. if we perform $a_i$ lending moves at each vertex $v_i$) then the new divisor is $D-L\sigma$. As a result, we can say that two divisors $D_1$ and $D_2$ are linearly equivalent if there exists some firing script $\sigma$ such that $D_2 = D_1-L\sigma$. Thus, the Dollar Game can be written as follows: is there some $\sigma \in \mathcal{M}(\Gamma)$ such that $D-L\sigma \geq 0$? 

\subsection*{Quotient Graphs and Weighted Chip-Firing}

Consider the group action of $\Z/2\Z$ on $\Gamma$ in Figure \ref{fig:quotientmap2} below, where the nontrivial element acts by reflection across the diagonal. In moving from a graph $\Gamma$ to its quotient $\Gamma/G$, there is certainly a natural map from $\Div(\Gamma)$ to $\Div(\Gamma/G)$; if $v_1,...,v_n$ are mapped to the same vertex under the quotient operation, then the degree of this new vertex is $\sum_{i=1}^n D(v_i)$.

\begin{figure}[ht]
    \centering
    \begin{tikzpicture}[node distance={15mm},main/.style = {draw,circle,fill, inner sep=2pt}]
\begin{scope}
\node[main,label = left:{$v_1$},label = {[blue]above:1}] (1) {};
\node[main,label = right:{$v_2$},label = {[blue]above:1}] (2) [right of = 1] {};
\node[main,label = left:{$v_3$},label = {[red]below:-3}] (3) [below of = 1] {};
\node[main,label = right:{$v_4$},label = {[blue]below:1}] (4) [right of = 3] {};
\draw (1) -- (2);
\draw (1) -- (3);
\draw (3) -- (4);
\draw (4) -- (2);
\draw (3) -- (2);
\end{scope}

\begin{scope}[xshift = 7cm]
\node[main,label = left:{$\{v_1,v_4\}$},label = {[blue]above:2}] (1) {};
\node[main,label = left:{$\{v_3\}$},label = {[red]below:-3}] (2) [below of = 1] {};
\node[main,label = right:{$\{v_2\}$},label = {[blue]above:1}] (3) [right of = 1] {};
\draw (1) -- (3);
\draw (1) -- (2);
\draw (2) -- (3);
\end{scope}

\draw[->] (2.7,-.8) -- (5.8,-.8);
\end{tikzpicture}
\caption{The natural map of a principal divisor from a graph to its quotient.}
    \label{fig:quotientmap2}
\end{figure}
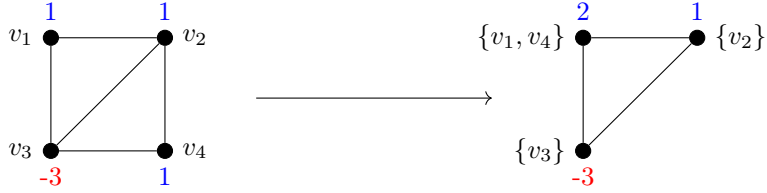

However, if chip-firing on the quotient graph is defined in the same way as the original graph, then this map is clearly not a homomorphism, as it maps a principal divisor to a non-principal divisor. The construction of weighted chip-firing provides a remedy for this; it defines chip-firing on a quotient graph in such a way that this natural map is a homomorphism.

\begin{definition}
A \emph{weighted quotient graph} is a quotient graph $\Gamma/G$ along with weight maps such that
\[w_V(O(v)) = |\text{Stab}(v)|,\; w_E(O(e)) = |\text{Stab}(e)|\]
for any $v \in V(\Gamma)$ and $e \in E(\Gamma)$. Typically both of these maps are denoted by $w$ unless specification is necessary.
\end{definition}

It is not immediately clear that the weight maps are well-defined. However, we know by the Orbit-Stabilizer theorem that for any pair of vertices or pair of edges which share an orbit, their stabilizers must be of the same order. We also know that $\text{Stab}(e) = \text{Stab}(h)$ is a subgroup of $\text{Stab}(r(h))$, which guarantees that $w(O(e))$ divides $w(O(v))$ whenever $O(v)$ is a root of $O(e)$ on the quotient graph $\Gamma/G$.

As an example, consider the quotient graph shown in Figure \ref{fig:quotientmap2}. We see that $v_2, v_3$, and the diagonal edge are fixed under the group action. Therefore their stabilizer is $\Z/2\Z$ and they are given weight 2. On the other hand, the outside edges, along with the vertices $v_1$ and $v_4$, only have stabilizer $\{1\}$ and thus have weight $1$. Therefore we can label the quotient graph with weights as in Figure 4. When it is clear what the weights are on a weighted graph, we may represent weights by vertices and edges of various sizes, with larger edges and vertices representing larger weights, rather than labelling all vertices and edges with weights (see Figure \ref{fig:representations}).

\begin{figure}[ht]
    \centering
    \begin{tikzpicture}[node distance={15mm},main/.style = {draw,circle,fill, inner sep=2pt}]
    \begin{scope}
        \node[main,label = above:1] (1) {};
        \node[main,label = right:2] (2) [below of = 1] {};
        \node[main,label = above:2] (3) [right of = 1] {};
        \draw (1) -- node[midway, above] {1} (3);
        \draw (1) -- node[midway, left] {1} (2);
        \draw (2) -- node[midway, below right] {2} (3);
    \end{scope}
    
    \begin{scope}[xshift = 4cm]
        \node[draw,circle,fill,inner sep=1pt] (1) {};
        \node[main] (2) [below of = 1] {};
        \node[main] (3) [right of = 1] {};
        \draw (1) -- (3);
        \draw (1) -- (2);
        \draw[line width = 1mm] (2) -- (3);
    \end{scope}
    \end{tikzpicture}
    \caption[Weighted graph representations]{Different representations of the same weighted graph}
    \label{fig:representations}
\end{figure}
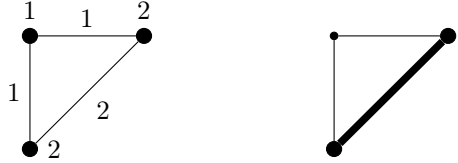

It will be useful to define one other quantity in relation to weighted graphs:

\begin{definition}
The \emph{charge} (or \emph{coweight}) of a vertex $v_i$ on a graph $\Gamma$ with vertices $v_1,...,v_n$ is defined as
\[c(v_i) := \dfrac{\lcm(w(v_1),...,w(v_n))}{w(v_i)},\]
where $w(v)$ denotes the weight of $v$. We refer to the quantity $\lcm(w(v_1),...,w(v_n))$ as the \emph{charge of} $\Gamma$, sometimes denoted $c(\Gamma)$.
\end{definition}

The charge of a vertex is dual to its weight, and bears much importance in the construction of chip-firing algorithms. On an unweighted graph, we assume all weights are 1, and therefore $c(v) = w(v) = 1$ for all vertices, and $c(\Gamma)= 1$.

Recall that the \emph{weighted valency} of a vertex is
\[\val(v) = \sum_{e \in E(v)} \dfrac{w(v)}{w(e)}.\]
Using this, we have a natural extension of chip-firing to weighted graphs.

\begin{definition}
Suppose $\Gamma$ is a weighted graph. Let $D,D' \in \Div(\Gamma)$ and $v \in V(\Gamma)$. Then $D'$ is obtained from $D$ by a \emph{lending move at v} if for $u \neq v$,
\[D'(u) = D(u)+\sum_{e \in E(u,v)} \dfrac{w(v)}{w(e)}, \quad \text{ and }\quad  D'(v) = D(v)-\val(v). \]
Similarly, $D'$ is obtained from $D$ by a \emph{borrowing move at v} if for $u \neq v$,
\[D'(u) = D(u)-\sum_{e \in E(u,v)} \dfrac{w(v)}{w(e)}, \quad \text{ and }\quad  D'(v) = D(v)+\val(v). \]
\end{definition}

In an intuitive sense, the weight of a vertex represents its ``strength'' in lending, while the weight of an edge represents its ``resistance'' to chip-firing. By this construction, weighted chip-firing looks notably different from unweighted chip-firing. For example, we see that, on our quotient graph with which we have been working, lending at $\{v_3\}$ sends out a total of three chips along two edges, while borrowing at $\{v_1,v_4\}$ still brings in a total of two chips along two edges:

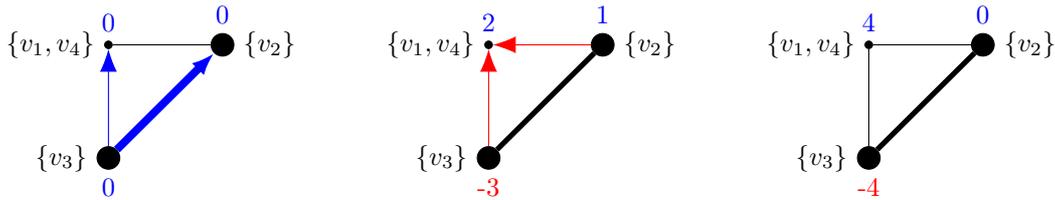
\begin{figure}[ht]
    \centering
    \begin{tikzpicture}[node distance={15mm},main/.style = {draw,circle,fill}]
    \begin{scope}
    \node[main,inner sep = 1pt,label = left:{$\{v_1,v_4\}$},label = {[blue]above:0}] (1) {};
    \node[main,inner sep = 3pt,label = left:{$\{v_3\}$},label = {[blue]below:0}] (2) [below of = 1] {};
    \node[main,inner sep = 3pt,label = right:{$\{v_2\}$},label = {[blue]above:0}] (3) [right of = 1] {};
    \draw (1) to (3);
    \draw[-{Latex[length=3mm]},blue] (2) to (1);
    \draw[-{Latex[length=3mm]},line width = 3pt, blue] (2) to (3);
    \end{scope}
    
    \begin{scope}[xshift = 5cm]
    \node[main,inner sep = 1pt,label = left:{$\{v_1,v_4\}$},label = {[blue]above:2}] (1) {};
    \node[main,inner sep = 3pt,label = left:{$\{v_3\}$},label = {[red]below:-3}] (2) [below of = 1] {};
    \node[main,inner sep = 3pt,label = right:{$\{v_2\}$},label = {[blue]above:1}] (3) [right of = 1] {};
    \draw[-{Latex[length=3mm]},red] (3) to (1);
    \draw[-{Latex[length=3mm]},red] (2) to (1);
    \draw[line width = 2pt] (2) to (3);
    \end{scope}
    
    \begin{scope}[xshift = 10cm]
    \node[main,inner sep = 1pt,label = left:{$\{v_1,v_4\}$},label = {[blue]above:4}] (1) {};
    \node[main,inner sep = 3pt,label = left:{$\{v_3\}$},label = {[red]below:-4}] (2) [below of = 1] {};
    \node[main,inner sep = 3pt,label = right:{$\{v_2\}$},label = {[blue]above:0}] (3) [right of = 1] {};
    \draw (3) to (1);
    \draw (2) to (1);
    \draw[line width = 2pt] (2) to (3);
    \end{scope}
    \end{tikzpicture}
    \label{fig:weighted-firing}
    \caption{Weighted lending and borrowing on a quotient graph.}
\end{figure}

While the chip-firing operations have become in some sense more complicated and asymmetric, we find that they are sufficient for our natural quotient map to serve as a homomorphism. As a result, we have analogous definitions of $\Prin(\Gamma),\Pic(\Gamma),$ and $\Jac(\Gamma)$ for weighted graphs.

\begin{definition}
Let $\Gamma$ be a weighted graph, $\Div(\Gamma)$ be the group of divisors on $\Gamma$, and $\Div^0(\Gamma)$ be the group of degree zero divisors. We define $\Prin(\Gamma)$ to be the group of \emph{principal divisors}, that is, divisors which are linearly equivalent (via \emph{weighted} chip-firing) to the zero divisor. We also define the \emph{Picard group} $\Pic(\Gamma)$ to be $\Div(\Gamma)/\Prin(\Gamma)$, and the \emph{Jacobian group} $\Jac(\Gamma)$ to be $\Div^0(\Gamma)/\Prin(\Gamma)$.
\end{definition}

\begin{definition}
Let $\Gamma$ be a weighted graph with $n$ vertices $v_1,...,v_n$. We define the \emph{weighted Laplacian} as the $n \times n$ matrix $L$ with
\[L_{i,j} = \begin{cases}
\val(v_i) \quad &\text{if } i=j \\
-\sum\limits_{e \in E(v_j,v_i)} \frac{w(v_j)}{w(e)} \quad &\text{if } i \neq j.
\end{cases}\]
\end{definition}

Note that the weighted Laplacian makes clear the asymmetry of weighted chip-firing; it is no longer necessarily a symmetric matrix. However, the weighted Laplacian still serves the same purpose as the Laplacian for unweighted graphs: it functions as a linear map $L:\Z^n \rightarrow \Z^n$ which inputs a firing script and outputs the corresponding principal divisor. We therefore can ask the same question as we do in the unweighted case: given a divisor $D$ on a graph $\Gamma$, is there a firing script $\sigma$ such that $D-L\sigma \geq 0$?

\section{The Modified Greedy Algorithm}

One of the basic methods of determining the ``winnability'' of a particular divisor on an unweighted graph is referred to as the Greedy Algorithm. The strategy is simple: repeatedly choose any vertex which is in-debt and borrow at that vertex until either: (i) the game is won, or (ii) you have borrowed at every vertex. The effectiveness of this algorithm relies on the fact that the kernel of the Laplacian on any unweighted graph is generated by the firing script $(1,1,...,1)$. This is not generally the case on weighted graphs, so to generalize the Greedy Algorithm we must first find the generator of the Laplacian for an arbitrary weighted graph.

\begin{lemma}\label{NetFire}
Let $\Gamma$ be a connected graph, and suppose $\sigma$ is a firing script on $\Gamma$ such that $L\sigma = \mathbf{0}$. Then $\sigma(u)w(u) = \sigma(v)w(v)$ for all $u,v \in V(\Gamma)$.
\end{lemma}
\begin{proof}
Define the function $f:V(\Gamma) \rightarrow \Z$ as
\[f(v) = \sigma(v)w(v).\]
Now suppose $f$ is not constant. Then since $V(\Gamma)$ is a finite set, there must exist some $v_0 \in V(\Gamma)$ such that $f(v_0) \geq f(v)$ for all $v \in V(\Gamma)$, and $v_0$ is adjacent to some vertex $v_1$ with $f(v_1)<f(v_0)$. But then if for each $i$, $v_i$ is adjacent to $v_0$ via $e_{1,i},...,e_{n_i,i}$, we have
\[\sum_{i=1}^k \sum_{j=1}^{n_i}\left( \dfrac{\sigma(v_0)w(v_0)}{w(e_{j,i})}-\dfrac{\sigma(v_i)w(v_i)}{w(e_{j,i})}\right) = \sum_{j=1}^k\sum_{j=1}^{n_i} \dfrac{f(v_0)-f(v_j)}{w(e_{j,i})} >0.\]

Therefore firing $\sigma$ removes at least one extra chip from $v_0$, so $\sigma \not \in \ker(L)$. This is a contradiction. So $f$ must be constant, which means that $\sigma(v)w(v) = \sigma(u)w(u)$ for all $u,v \in V(\Gamma)$.
\end{proof}

\begin{theorem}\label{thm:script}
Let $\Gamma$ be a graph with Laplacian $L$, and recall that $c(v)$ is the charge of $v$. Then the firing script defined by
\[\sigma(v) = c(v)\]
for all $v \in V(\Gamma)$ generates $\ker(L)$.
\end{theorem}
\begin{proof}
First, we show that this firing is in $\ker(L)$: let $v_1,v_2 \in V(\Gamma)$ and $e_1,\dots,e_n \in E(\Gamma)$ be any adjacent vertices and their connecting edges. Suppose we fire $v_1$ and $v_2$ a total of $c(v_1)$ and $c(v_2)$ times respectively. Then
\[c(v_1) \cdot \dfrac{w(v_1)}{w(e_i)} = \dfrac{c(\Gamma)w(v_1)}{w(v_1)w(e_i)} = \dfrac{c(\Gamma)}{w(e_i)} = \dfrac{c(\Gamma)w(v_2)}{w(v_2)w(e_i)} = c(v_2) \cdot \dfrac{w(v_2)}{w(e_i)},\]
for each $i$, which means that the number of chips sent to $v_2$ by $v_1$ along each $e_i$ is the same as the number of chips sent to $v_1$ by $v_2$ along $e_i$. Since this is the case for all $v_1,v_2 \in V(\Gamma)$ and any connecting edge $e_i \in E(\Gamma)$, it follows that $L\sigma = 0$.

Now to show that $\ker(L) = \langle\sigma\rangle$, suppose $\sigma' \in \ker(L)$. We have that $\sigma'(v)w(v) = \sigma'(u)w(u)$ for all $u,v \in V(\Gamma)$ by Lemma \ref{NetFire}. But that means that $w(u)$ divides $\sigma'(v)w(v)$ for all $u,v \in V(\Gamma)$. Thus we have two cases. If $w(u)=w(v)$ for all $u,v \in V(\Gamma)$, then $\sigma'(u) = \sigma'(v)$ for all $u,v \in V(\Gamma)$ and $\sigma' = (k,k,...,k) = k(1,1,...,1) = k(c(v_1),...,c(v_n))$. If there exist $u,v \in V(\Gamma)$ such that $w(u) \neq w(v)$, then we have that $c(v)$ divides $\sigma'(v)$ for all $v$, so $\sigma'(v) = kc(v)$ for some constant $k$.
\end{proof}

\begin{corollary}
    Let $\Gamma$ be a connected graph with $n$ vertices and Laplacian $L$. Then $\text{rank}(L) = n-1$.
\end{corollary}
\begin{proof}
    This follows from $\ker(L) = \langle \sigma\rangle$ for $\sigma(v) = c(v)$.
\end{proof}

\begin{corollary}
The Jacobian of a finite weighted graph is finite.
\end{corollary}
\begin{proof}
Recall that $L$ is a linear operator which takes the set of firing scripts $\mathcal{M}(\Gamma)$ to $\Prin(\Gamma)$, which is exactly the kernel of the natural quotient map from $\Div^0(\Gamma)$ to $\Jac(\Gamma)$. If $\sigma(v):=c(v)$ for all $v$, then we have shown that $\sigma$ generates $\ker(L)$, so we have a map $\varphi:\Z \rightarrow \mathcal{M}(\Gamma)$ defined by $k \mapsto k\sigma$ such that $\text{Im}(\varphi) = \ker(L)$. Therefore we can construct the exact sequence
\[0 \rightarrow \Z \xrightarrow{\varphi} \mathcal{M}(\Gamma) \xrightarrow{-L\sigma} \Div^0(\Gamma) \rightarrow \Jac(\Gamma) \rightarrow 0.\]
We know that $\Div^0(\Gamma)$ has rank $n-1$, since any such divisor is of the form $(k_1,...,k_n,-(k_1+\cdots+k_n))$. But $\mathcal{M}(\Gamma)$ has rank $n$, so
\[\text{rank}(\Jac(\Gamma)) = -\text{rank}(\Z)+ \text{rank}(\mathcal{M}(\Gamma))-\text{rank}(\Div^0(\Gamma)) = -1 + n - (n-1) = 0.\]

Thus $\Jac(\Gamma)$ must be a finite group.
\end{proof}

Now that we have the generator for $\ker(L)$ for any weighted graph, we can construct a generalized Greedy Algorithm, the only modification being to terminate at the new ``identity script''. 

\begin{algorithm}[ht]
\SetAlgoLined
\KwResult{This algorithm returns \textsc{True} if $D$ is winnable, else \textsc{False}.}
$m(v_i)=0$ for all $v_i$\;
\While{D is not effective}{
\eIf{$\exists v \in V$ s.t. $m(v)<c(v)$}{
choose any vertex $v_i \in V$ s.t. $D(v_i)<0$\;
borrow at $v_i$\;
set $m(v_i)=m(v_i)+1$\;}{
return \textsc{False}}
}
return \textsc{True}
\caption{The Modified Greedy Algorithm}
\end{algorithm}

\begin{proposition}
Let $D$ be a divisor on $\Gamma$. The Modified Greedy Algorithm is a complete algorithm with termination, and it returns \textsc{True} on input $D$ if and only if $D$ is winnable.
\end{proposition}
\begin{proof}

First, suppose $D$ is winnable, that is, suppose there exists $D' \in \Div(G)$ such that $D \sim D'$ and $D'$ is effective. Now let $\sigma$ be a firing script such that $D' = D-L\sigma$. Using Theorem \ref{thm:script} we may assume that $\sigma \leq 0$ and that there exists some $v \in V$ such that $|\sigma(v)|< c(v)$ by subtracting an appropriate multiple of the identity script. Applying the algorithm, we see the following:
\begin{itemize}
    \item If $D$ is effective, we are done.
    \item If $D$ is not effective, each loop of the algorithm takes some vertex $u$ and replaces the borrowing script $\sigma(u)$ with $\sigma(u)+1$, that is, the new $\sigma$ requires one less borrowing move from $u$. When $\sigma(u) = 0$, $u$ must necessarily be out of debt since no further borrowing move would add chips to $u$. Therefore the algorithm must terminate at the winning solution, since there is a finite number of valid moves before $\sigma(u) = 0$ for all $u \in V(\Gamma)$.
\end{itemize}
Now suppose $D$ is unwinnable. The algorithm only returns \textsc{True} if some sequence of borrowing moves turns $D$ into an effective divisor, so either the algorithm returns \textsc{false} or it does not terminate. Suppose the algorithm does not terminate. Then in each loop of the algorithm there must always be some vertex which is in debt. Since $V(\Gamma)$ is finite, there must exist some nonempty set of vertices $A$ such that $v \in A$ borrows infinitely. If $v \in A$, and $u$ is adjacent to $v$, then $v$ must borrow infinitely from $u$. So in order to stay out of debt, $u$ must also borrow infinitely, so $u \in A$. Since $\Gamma$ is connected, that means $A = V(\Gamma)$. Since every vertex borrows infinitely, there must be some iteration of the algorithm at which point $m(v) \geq c(v)$ for all $v \in V(\Gamma)$, which means the algorithm must terminate with a return value of \textsc{false}. 

Termination and completion follow trivially from the above proof.
\end{proof}

\section{Burning and $q$-Reduced Divisors}

The original formulation of Dhar's Burning Algorithm \cite{dhar} provided an alternative strategy to greed for winning the Dollar Game, which can be called generosity. Let $\Gamma$ be a graph, weighted or unweighted. It is easy to prove that, for any vertex $q$ and any divisor $D$, we can find a linearly equivalent divisor $D'$ such that $D'(v) \geq 0$ for all $v \neq q$. After $q$ has been ``generous'' to the other vertices, it's time for them to pay it back. We wish to know: is there any firing script on $V(\Gamma)\backslash\{q\}$ which brings $q$ out of debt without putting any other vertex back into debt? Dhar's Burning Algorithm provides a method of checking $2^{|V|}-1$ potential firing scripts in linear time.

\begin{center}
    \fbox{\begin{minipage}{15cm}
\noindent\textbf{Dhar's Burning Algorithm:} Let $q$ be a vertex in $V(\Gamma)$, and $D$ a divisor on a $\Gamma$ such that $D(v) \geq 0$ for all $v \neq q$. Let $S = \{q\}$. 

\begin{itemize}[wide=0.5em, leftmargin =*, nosep, before = \leavevmode\vspace{-\baselineskip}]

    \item [1.] Let $\sigma$ be the firing script which fires all vertices not in $S$.
    \item [2.] If $D' = D-L\sigma$ has $D'(v) \geq 0$ for all $v \neq q$, then return $\sigma$. If not, then let $S = S \cup S'$, where $S'$ the set of vertices $v$ such that $D'(v) < 0$. 
    \item [3.] If $S = V(\Gamma)$, then return the zero script. If not, then return to step (1).
\end{itemize} 
\end{minipage}}
\end{center}

Thus we can determine the winnability of a divisor on an unweighted graph in linear time, first by bringing a vertex $q$ out of debt, then repeatedly applying Dhar's Algorithm until either (i) the game is won, or (ii) Dhar's Algorithm returns the zero script, signalling that any firing script on $V(\Gamma)\backslash\{q\}$ will send some other vertex into debt. Figure \ref{fig:dhar2} shows a example run of the algorithm.

\begin{figure}[ht]
    \centering
    \begin{tikzpicture}[node distance={15mm},main/.style = {draw,circle,fill, inner sep=2pt}]
    \begin{scope}
    \node[main,label = {[red]above:-1},red] at (90:1cm) (1) {};
    \node[main,label = {[blue]above:1}] at (162:1cm) (2) {};
    \node[main,label = {[blue]left:0}] at (234:1cm) (3) {};
    \node[main,label = {[blue]right:1}] at (306:1cm) (4) {};
    \node[main,label = {[blue]above:0}] at (18:1cm) (5) {};
    
    \draw[-{Latex[length=3mm]},line width = 2pt,red] (1) to (2);
    \draw[-{Latex[length=3mm]},line width = 2pt,red] (1) to (3);
    \draw[-{Latex[length=3mm]},line width = 2pt,red] (1) to (5);
    \draw (3) to (2);
    \draw (3) to (4);
    \draw (5) to (4);
    \end{scope}
    
    \begin{scope}[xshift=5cm]
    \node[main,label = {[red]above:-1},red] at (90:1cm) (1) {};
    \node[main,label = {[blue]above:1}] at (162:1cm) (2) {};
    \node[main,label = {[blue]left:0},red] at (234:1cm) (3) {};
    \node[main,label = {[blue]right:1}] at (306:1cm) (4) {};
    \node[main,label = {[blue]above:0},red] at (18:1cm) (5) {};
    
    \draw[-{Latex[length=3mm]},line width = 2pt,red] (1) to (2);
    \draw[-{Latex[length=3mm]},line width = 2pt,red] (1) to (3);
    \draw[-{Latex[length=3mm]},line width = 2pt,red] (1) to (5);
    \draw[-{Latex[length=3mm]},line width = 2pt,red] (3) to (2);
    \draw[-{Latex[length=3mm]},line width = 2pt,red] (3) to (4);
    \draw[-{Latex[length=3mm]},line width = 2pt,red] (5) to (4);
    \end{scope}
    
    \begin{scope}[xshift=10cm]
    \node[main,label = {[red]above:-1},red] at (90:1cm) (1) {};
    \node[main,label = {[blue]above:1},red] at (162:1cm) (2) {};
    \node[main,label = {[blue]left:0},red] at (234:1cm) (3) {};
    \node[main,label = {[blue]right:1},red] at (306:1cm) (4) {};
    \node[main,label = {[blue]above:0},red] at (18:1cm) (5) {};
    
    \draw[-{Latex[length=3mm]},line width = 2pt,red] (1) to (2);
    \draw[-{Latex[length=3mm]},line width = 2pt,red] (1) to (3);
    \draw[-{Latex[length=3mm]},line width = 2pt,red] (1) to (5);
    \draw[-{Latex[length=3mm]},line width = 2pt,red] (3) to (2);
    \draw[-{Latex[length=3mm]},line width = 2pt,red] (3) to (4);
    \draw[-{Latex[length=3mm]},line width = 2pt,red] (5) to (4);
    \end{scope}
    \draw[->,line width=1pt] (2,0) to (3,0);
    \draw[->,line width=1pt] (7,0) to (8,0);
    \end{tikzpicture}
    \caption[The burning algorithm on an unweighted graph]{A run of Dhar's Burning Algorithm which returns the zero script. Once the number of arrows directed towards a vertex exceeds its degree, it ``burns'' and sends arrows along its adjacent edges.}
    \label{fig:dhar2}
\end{figure}
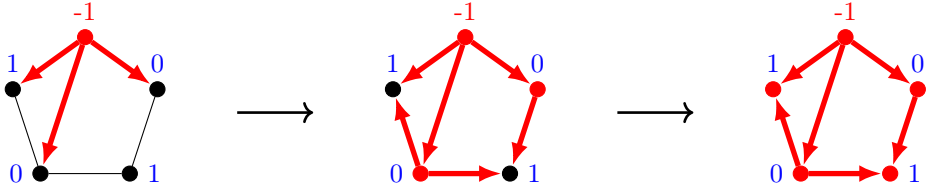

A divisor which burns completely (i.e. returns the empty set) is what we refer to as a $q$-reduced divisor, on which the only firing scripts which keep all non-$q$ vertices out of debt are firing scripts which fire $q$.

\begin{definition}
Let $D$ be a divisor on a graph $\Gamma$. A \emph{legal firing script on} $D$ is a firing script $\sigma$ such that $\sigma(v)\geq0$ for all $v$ (that is, $\sigma$ is a collection of lending moves) and $(D-L\sigma)(v) < 0$ only if $D(v)<0$.
\end{definition}

\begin{definition}
Let $\Gamma$ be a graph, and $q \in V(\Gamma)$. A \emph{q-effective divisor} is a divisor $D$ on $\Gamma$ such that $D(v)\geq 0$ for all $v \in V(\Gamma)\backslash\{q\}$.
\end{definition}

\begin{definition}
Let $\Gamma$ be an unweighted graph, and $D$ a $q$-effective divisor on $\Gamma$. We say that $D$ is \emph{$q$-reduced} if every legal firing script on $D$ has $\sigma(q)>0$.
\end{definition}

In particular, a divisor on an unweighted graph is $q$-reduced if it can't push any chips closer to $q$ without sending another vertex into debt.

Dhar's Algorithm relies on this result on unweighted graphs: given a vertex $q \in V(\Gamma)$, for any firing script $\sigma$, there exists another firing script $\sigma'$ such that $L\sigma = L\sigma'$ and $\sigma'(q) = 0$. This does not generally hold for weighted graphs. Since the kernel-generating script $\sigma(v) = c(v)$ on an unweighted graph is $(1,1,...,1)$, we can replace any firing script $\tau = (\tau(v_1),...,\tau(q),...,\tau(v_n))$ by the equivalent firing script $\tau' = (\tau(v_1)-\tau(q),...,0,...,\tau(v_n)-\tau(q))$. It is clear, though, that when the kernel-generating script for $\Gamma$ is $(c(v_1),...,c(v_n))$, then a firing script $\sigma$ only has an equivalent form $\sigma'$ with $\sigma'(q)=0$ if $c(q)$ divides $\sigma(q)$.

If we seek to construct a general form of Dhar's Burning Algorithm, we must first generalize the idea of a $q$-reduced divisor. Thinking of the continued iteration of Dhar's Algorithm as an attempt to amass as much wealth as possible in a neighborhood of $q$, we are led to the concept of a divisor which is a dead end, where it is not possible to push any more chips to $q$ without putting another vertex in debt. 
\begin{definition}\label{qred}
A divisor $D \in \Div(G)$ is called \emph{$q$-reduced} if the following properties hold: 

\begin{itemize}[wide=0.5em, leftmargin =*, nosep, before = \leavevmode\vspace{-\baselineskip}]
    \item [(i)] $D$ is $q$-effective, i.e. $D(v) \geq 0$ for all $v \neq q$,
    \item [(ii)] any linearly equivalent $q$-effective divisor $D'$ is such that $D'(q) \leq D(q)$, and
    \item [(iii)] if $D' \in [D]$ is such that $D'(q) = D(q)$, then any effective firing script $\sigma$ which sends $D$ to $D'$ has $\sigma(q)\neq 0$.
\end{itemize}
\end{definition}
Thus we can generalize Dhar's Burning Algorithm by modifying in two places: first, instead of firing each vertex once and removing the vertices which go into debt from the next script, we fire each vertex $v \neq q$ a total of $c(v)$ times, and if it goes into debt, the next script fires $v$ a total of $c(v)-1$ times. The second modification is that we must account for scripts which do fire $q$, and so we run the central algorithm $c(q)$ times, considering scripts of the form $(\sigma(v_1),...,\sigma(q)=f,...,\sigma(v_n))$ for some fixed $f$ in each iteration. Each choice of $f$ gives us a candidate for a legal firing, and to optimize we choose one which leaves the most chips at $q$.

\begin{algorithm}[ht]
\SetAlgoLined
\KwResult{This algorithm inputs a $q$-effective divisor $D$, and returns a legal firing script $\sigma$, with $\sigma = \mathbf{0}$ iff $D$ is $q$-reduced.}
$\sigma(v) = c(v)$ if $v \neq q$\;
$f = 0$ (the number of times $q$ has been fired)\;
\While{$f < c(q)$}{
$\sigma(q) = f$\;
\While{$\sigma(v) \neq 0$ for some $v \neq q$}{
$m_i = (D-L \cdot \sigma)(v_i)$\;
\If{$m_i \geq 0$ $\forall i$}{
break\;
}
\For{$v_i:m_i<0$}{
$\sigma(v_i) = \sigma(v_i)-1$\;}
}
$\sigma_f = \sigma$\;
fire $q$ ($f = f+1$)\;
}
$f=0$\;
\For{$f<c(q)$}{
\If{$[L \cdot \sigma_f](q) \leq [L \cdot \sigma_k](q) \;\forall k$}{
return $D_f$\;
}
$f = f+1$\;
}
\caption{The Modified Burning Algorithm}
\end{algorithm}

\begin{proposition}
    The Modified Burning Algorithm is a complete algorithm with termination which returns a legal firing script, and it returns $\mathbf{0}$ if and only if $D$ is $q$-reduced.
\end{proposition}
\begin{proof}
    It is clear that if the algorithm returns a script, then the script is legal.

    We first show termination. The outer while loop has only a finite set of stages ($0\leq f < c(q)<\infty$). For the inner while loop, we begin with the list $(\sigma(v_1),...,\sigma(v_n)) = (c(v_1),...,c(v_n))$, and in iterating, we either terminate with $m_i\geq 0$ for all $i$, or we have some $i$ for which $\sigma(v_i)$ is replaced by $\sigma(v_i)-1$. Note that if $\sigma(v_i) = 0$, then $m_i \geq 0$ since firing $v_i$ is the only way for $v_i$ to lose any chips (as the scripts considered only consist of lending moves), and by assumption $D(v_i)\geq 0$. Thus there must be a finite number of stages in the inner while loop, and so the algorithm terminates.

    Next, we show completeness. Suppose $D$ is not $q$-reduced, i.e. there exists some legal non-zero firing script $(\sigma(v_1),...,\sigma(v_n))$ on $D$, and suppose the algorithm returns $\mathbf{0}$. We may assume WLOG that $0\leq \sigma(v_i)\leq c(v_i)$ for all $v_i$ and that $\sigma(v_i)>0$ for some $v_i$. (To verify this, let $\sigma$ be a legal firing script such that $\sigma(v_{i})>c(v_{i})$ for some collection of vertices $v_1,...v_m$ and $\sigma(v_i)\leq c(v_i)$ otherwise. Now let $k$ be the minimum value such that $\sigma(v_i)-kc(v_i)\leq c(v_i)$ for all $i$. Then the script
    \[\tau = (\max\{\sigma(v_i)-kc(v_i),0\})\]
    is legal with $\tau(v_i)\leq c(v_i)$ for all $i$.) Thus the algorithm checks the script $(c(v_1),...,c(v_n))$ which dominates $\sigma$. Let $\tau$ be the least script checked which dominates $\sigma$. Then we have that $\tau(v_i) = \sigma(v_i)$ for at least one vertex $v_i$. But running the inner loop on $\tau$, we have that $m_i = (D-L\cdot\tau)(v_i) \geq (D-L\cdot\sigma)(v_i)\geq 0$, since extra lending moves at other vertices only contribute positively to $m_i$. Thus the script which replaces $\tau$ must also dominate $\sigma$, a contradiction. So the algorithm must return some legal non-zero firing script.

    The proof of completeness shows the forward direction of the theorem; to see the backward direction, suppose $D$ is $q$-reduced. Then for any possible firing set $\sigma$, we have that $m_i = [L \cdot \sigma]_i + \deg(v_i)<0$, in which case the while $\sigma \neq \mathbf{0}$ loop would not break until $\sigma = \mathbf{0}$, regardless of how many times $q$ is fired. That means that $\sigma_f = \mathbf{0}$ for all $f$, so the algorithm returns $\mathbf{0}$. By contrapositive we are done
\end{proof}

Just as with Dhar's Algorithm in the unweighted case, this modified form of Dhar's Algorithm is incredibly powerful, considering
\[\prod_{v \in V(\Gamma)} (c(v)+1)-1 \geq 2^{|V|}-1\]
potential firing scripts in at most
\[c(q)\left(\sum_{v \in V(\Gamma\backslash\{q\})} c(v) \right)\leq c(q)c(\Gamma)(|V(\Gamma)|-1) = O(|V|) \text{ steps}.\]

We find that some preliminary results demonstrated in \cite{perkinson} regarding $q$-reduced divisors can be extended to weighted graphs fairly easily using this algorithm.

\begin{definition}
Let $D,D' \in \Div(\Gamma)$, and let $v_1 \prec ... \prec v_n$ be a tree ordering compatible with a spanning tree $T$ rooted at $q$. We say that $D' \prec D$ if (i) $\deg(D')<\deg(D)$, or (ii) $\deg(D') = \deg(D)$ and $D'(v_k)>D(v_k)$ for the smallest index $k$ such that $D'(v_k) \neq D(v_k)$. We say that $\prec$ is a \emph{tree ordering rooted at $q$}. 
\end{definition}

\begin{lemma}\label{exist}
Let $D \in \Div(G)$, and fix $q \in V$. Then there exists at least one $q$-reduced divisor linearly equivalent to $D$.
\end{lemma}
\begin{proof}
Fix a tree ordering $v_1 \prec ... \prec v_n$ compatible with a spanning tree $T$ rooted at $q$. Now each vertex $v \neq q$ has a unique neighbor $\epsilon(v) \in T$ such that $\epsilon(v) \prec v$. Identically to the unweighted case, bring $v_n,v_{n-1},...,v_2$ out of debt successively by lending to $v_i$ from $\epsilon(v_i)$.

So we can assume that $D(v) \geq 0$ for all $v \neq q$. If $D$ is not $q$-reduced, we can apply the modified burning algorithm to find a legal set firing $\sigma$ such that the divisor $D' = D -L\sigma$ is still $q$-effective and $D'$ is strictly smaller than $D$ with respect to the tree ordering rooted at $q$. By repeating this process, we get a sequence of divisors $D=D_1,D_2,...$. Split $D_i$ into $E_i+k_iq$, where $E_i(v) = D_i(v)$ if $v\neq q$ and $E_i(q)=0$. We have that $\deg(E_i) \geq \deg(E_{i+1}) \geq 0$ for all $i$ (as the algorithm will only choose a set-firing $\sigma$ if it leaves $q$ with at least as many chips as it already had). So there can only be finitely many distinct divisors in the sequence. But each divisor in the sequence is distinct by definition, so the sequence must stop at a $q$-reduced divisor.
\end{proof}

\begin{corollary}
Let $D \in \Div(G)$, and let $D'$ be a $q$-reduced divisor linearly equivalent to $D$. Then $D$ is winnable if and only if $D'(q) \geq 0$.
\end{corollary}
\begin{proof}
Suppose $D$ is winnable, and let $E \sim D$. From $E$, perform all legal set firings on $V\backslash\{q\}$, arriving at a $q$-reduced divisor $E'\geq 0$. By definition, $D'(q)=E'(q) \geq 0$ for any $q$-reduced divisor $D' \in [D]$. The converse is true by definition.
\end{proof}

In the unweighted case, there is a natural bijection between ordered pairs $(D,q)$ and $q$-reduced divisors (that is, each divisor class $[D]$ has a unique $q$-reduced representative for each vertex $q$). However, in the weighted case, we find that there may be a multitude of linearly equivalent divisors which satisfy Definition \ref{qred}. As an example, the divisors $(1,0,0,-1)$ and $(0,1,0,-1)$ on the graph of weights 1 and 2 in Figure 8 are both linearly equivalent and $q$-reduced.

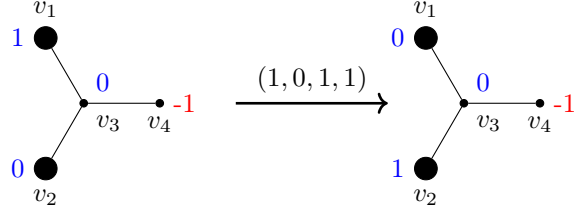
\begin{figure}[ht]
    \centering
    \begin{tikzpicture}[node distance={15mm},main/.style = {draw,circle,fill, inner sep=1pt}]   
    \begin{scope}
    \node[main,label = {[blue]above right:0},label = below right:{$v_3$}] at (0,0) (1) {};
    \node[draw,circle,fill,inner sep=3pt,label = {[blue]left:1},label = above:{$v_1$}] at (120:1cm) (2) {};
    \node[draw,circle,fill,inner sep=3pt,label = {[blue]left:0},label = below:{$v_2$}] at (240:1cm) (3) {};
    \node[main,label = {[red]right:-1},label = below:{$v_4$}] at (0:1cm) (4) {};
    
    \draw (1) to (2);
    \draw (1) to (3);
    \draw (1) to (4);
    \end{scope}
    
    \begin{scope}[xshift=5cm]
    \node[main,label = {[blue]above right:0},label = below right:{$v_3$}] at (0,0) (1) {};
    \node[draw,circle,fill,inner sep=3pt,label = {[blue]left:0},label = above:{$v_1$}] at (120:1cm) (2) {};
    \node[draw,circle,fill,inner sep=3pt,label = {[blue]left:1},label = below:{$v_2$}] at (240:1cm) (3) {};
    \node[main,label = {[red]right:-1},label = below:{$v_4$}] at (0:1cm) (4) {};
    
    \draw (1) to (2);
    \draw (1) to (3);
    \draw (1) to (4);
    \end{scope}
    \draw[->,line width=1pt] (2,0) to node[midway,above] {$( 1,0,1,1)$} (4,0);
    \end{tikzpicture}
    \caption{These divisors are linearly equivalent, but are each $q$-reduced.}
    \label{fig:doubleq}
\end{figure}

\begin{definition}
Let $D,D' \in \Div(\Gamma)$, $q$ in $V(\Gamma)$. We say $D$ and $D'$ are \emph{$q$-linearly equivalent} and write $D \sim_q D'$ if there exists a firing script $\sigma$ which takes $D$ to $D'$ such that $\sigma(q) = 0$. We call an equivalence class under this relation a \emph{$q$-class} and denote it by $[D]_q$.
\end{definition}

It is clear that this is a finer relation than regular linear equivalence. There are $c(q)$ total $q$-classes in each divisor class, and in particular, we know that given $q$, the set of divisor classes and the set of $q$-classes are equivalent for all $q$ if and only if $c(q) = 1$ for all $q$. In general though, what is the relationship between $q$-classes and the set of $q$-reduced divisors?

We know that if two $q$-classes $[D]_q$ and $[D']_q$ (where $D \sim D'$) each have a $q$-reduced divisor (say $E$ and $E'$, respectively), then $E(q) = E'(q)$ by definition. Therefore the firing set which sends $E$ to $E'$ must send an equal amount of chips to and from $q$. This leads us to the definition of the local charge of a vertex $q$.

\begin{definition}
Let $q \in V(\Gamma)$. The \emph{local charge} of $q$ is
\[c_\ell(q) = \dfrac{\lcm\left(\gcd\limits_{v \in V(\Gamma)} \left( \sum\limits_{e \in E(q,v)}\frac{w(v)}{e}\right),\val(q)\right)}{\val(q)},\]
where $\val(q)$ is the weighted valency of $q$.
\end{definition}

\begin{lemma}
    Let $D, D'$ be $q$-reduced divisors with $D'=D-L\sigma$. Then $c_\ell(q)|\sigma(q)$.
\end{lemma}

\begin{proof}
    Let $D, D'$ be $q$-reduced with $D'=D-L\sigma$ such that $\sigma(q)<c_\ell(q)$. Then the number of chips sent from $q$ is 
    \[k\val(q)<c_\ell(q)\val(q) = \lcm\left(\gcd\limits_{v \in V(\Gamma)} \left( \sum\limits_{e \in E(q,v)}\frac{w(v)}{e}\right),\val(q)\right).\]
    Now let $n$ denote the number of chips sent to $q$ from other vertices.
    By definition, $D(q)=D'(q)$, so we have that
    \[k\val(q)=n =\sum_{v \in V} \sigma(v) \left(\sum_{e \in E(q,v)} \dfrac{w(v)}{e}\right)\]
    But then $\val(q)|n$, and by its summation representation, 
    \[\gcd\limits_{v \in V(\Gamma)} \left( \sum\limits_{e \in E(q,v)}\frac{w(v)}{e}\right)\big|n,\quad\text{so}\quad\lcm\left(\gcd\limits_{v \in V(\Gamma)} \left( \sum\limits_{e \in E(q,v)}\frac{w(v)}{e}\right),\val(q)\right)\big|n = k\val(q).\]
    Dividing by $\val(q)$, we have that $c_\ell(q)|k$.
\end{proof}

\begin{corollary}\label{cor:q-class}
Let $D \in \Div(G)$, and let $\{[D_1]_q,...,[D_n]_q\}$ be the set of distinct $q$-classes which are subsets of $[D]$, where $D_i$ is the divisor obtained from $D$ by firing $q$ exactly $i$ times. Then 

\begin{itemize}[wide=0.5em, leftmargin =*, nosep, before = \leavevmode\vspace{-\baselineskip}]
    \item [(i)] Every $q$-class contains at most one $q$-reduced divisor, and
    \item [(ii)] If $[D_j]_q$ contains a $q$-reduced divisor, then any $q$-class containing a $q$-reduced divisor is of the form $[D_{j+kc_\ell(q)}]_q$ for some integer $k$.
\end{itemize}
\end{corollary}
\begin{proof}
Let $E,E' \in [D]$ be distinct $q$-reduced divisors. Then directly from the definition, we have that $[E]_q \neq [E']_q$, since any effective firing script $\sigma$ sending $E$ to $E'$ or $E'$ to $E$ must have $\sigma(q)>0$. Therefore each $[D_i]_q$ contains at most one $q$-reduced divisor.

Now suppose $[D_i]_q$ and $[D_j]_q$ each contain $q$-reduced divisors, labelled $E$ and $E'$ respectively. Then $E(q) = E'(q)$, which means that any set firing $\sigma$ which sends $E$ to $E'$ has $\sigma(q) = kc_\ell(q)$ for some $k \in \Z$, which means that $j = i + kc_\ell(q)$; that is, $[D_j]_q = [D_{i+kc_\ell(q)}]_q$.
\end{proof}

\begin{corollary}
Let $\Gamma$ be a weighted graph such that $w(v) \in \{1,2\}$ for all $v \in V(\Gamma)$. Then $\Gamma$ has at least one vertex $q$ such that there is a unique $q$-reduced form for each divisor $D$ on $\Gamma$.
\end{corollary}
\begin{proof}
Consider any vertex $q \in \Gamma$ with $w(q) = 2$. If such a vertex does not exist, then the graph is unweighted and the uniqueness of $q$-reduced divisors is clear. Otherwise, $c(q) = 1$, and since $c_\ell(q)$ divides $c(q)$, therefore $c_\ell(q) = c(q) = 1$. Thus any divisor $D$ on $\Gamma$ must have a unique $q$-reduced form.
\end{proof}

Theorem \ref{cor:q-class} provides an upper bound for the number of linearly equivalent $q$-reduced divisors for any given pair $(\Gamma,q)$; in particular, we know that if $c_\ell(q) = c(q)$, then $q$ has a unique $q$-reduced divisor. Note that there need not be a $q$-reduced divisor in every $q$-class of the form $[D_{j+kc_\ell(q)}]_q$; consider the following example, where $q$ is the bottom-right vertex:

\begin{figure}[ht]
    \centering
    \begin{tikzpicture}[node distance={15mm},main/.style = {draw,circle,fill, inner sep=2pt}]
    \begin{scope}
    \node[main,label = {[blue]right:1}] at (45:1cm) (1) {};
    \node[draw, circle, inner sep=4pt,fill,label = {[blue]left:0}] at (135:1cm) (2) {};
    \node[main,label = {[blue]left:0}] at (225:1cm) (3) {};
    \node[main,label = {[red]right:-1}] at (315:1cm) (4) {};
    \node[main] at (135:1cm) (5) {};
    \draw (1) -- (2);
    \draw (2) -- (3);
    \draw (3) -- (4);
    \draw (4) -- (1);
    \end{scope}
    
    \begin{scope}[xshift = 6cm]
    \node[main,label = {[blue]right:0}] at (45:1cm) (1) {};
    \node[draw, circle, inner sep=4pt,fill,label = {[blue]left:1}] at (135:1cm) (2) {};
    \node[main,label = {[blue]left:1}] at (225:1cm) (3) {};
    \node[main,label = {[red]right:-2}] at (315:1cm) (4) {};
    \node[main] at (135:1cm) (5) {};
    \draw (1) -- (2);
    \draw (2) -- (3);
    \draw (3) -- (4);
    \draw (4) -- (1);
    \end{scope}
    \draw[->] (2,0) -- (4,0);
    \end{tikzpicture}
    \caption[The local charge on a weighted graph]{Although $c(q)/c_\ell(q) = 2$, there is only one $q$-reduced form for this divisor.}
    \label{fig:qclass}
\end{figure}
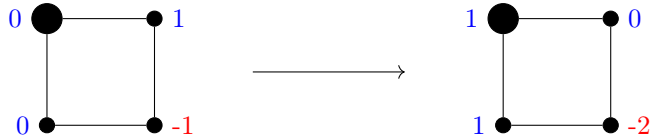

\section{Maximally Unwinnable Divisors}

In this section for simplicity we act under the assumption that each graph has some vertex $q$ of charge 1. This allows us to consider only a single iteration of the outer while-loop in the Modified Burning Algorithm, and guarantees that each class has exactly one $q$-reduced form. It is clear that these results may be extended to the general case with some care.

\begin{definition}
A divisor $D$ is called \emph{maximally unwinnable} if $D$ is unwinnable and $D+v$ is winnable for all $v \in V$.
\end{definition}

\begin{definition}
An \emph{orientation} of a graph $\Gamma$ is an assignment of a direction to each edge. An \emph{acyclic orientation} of a graph $\Gamma$ is an orientation such that following the direction of the orientation will never form a cycle.
\end{definition}

In the unweighted case, a critical result is the existence of a bijection between acyclic orientations of a graph with a unique source vertex and its set of maximal unwinnable divisors. In particular, we know that given an acyclic orientation $\mathcal{O}$ with a unique source vertex on $\Gamma$, the divisor defined by
\[D(\mathcal{O})(v) = \text{indeg}_\mathcal{O}(v)-1\]
is a maximally unwinnable divisor, where indeg$(v)$ is the number of arrows entering the vertex $v$ on $\mathcal{O}$. This follows from the fact that Dhar's Algorithm has a unique run on any maximally unwinnable divisor. The notion of orientation is not as useful on weighted graphs, where the modified algorithm often results in burning edges multiple times. Instead, we introduce the notion of words on a graph.

\begin{definition}
Let $\Gamma$ be a weighted graph. We say that $W = W(1)W(2)...W(L)$ is a \emph{word} on $\Gamma$ if $W(i) \in V(\Gamma)$ for all $i$ and $|\Ind_W(v)| = c(v)$ for all $v$, where the \emph{index set} of $v$, $\Ind_W(v)$, is defined by
\[\Ind_W(v) = \{n \in \{1,...,L\}:W(n) = v\}.\]
We also denote by $k^W_v(n):\N \rightarrow \N$ the step function of $v$ on $W$, that is,
\[k^W_v(n) = \#\{m \in \Ind_W(v):m<n\}.\]
\end{definition}

To each unwinnable divisor we can naturally construct an associated word by tracking the order in which each vertex burns.  Consider the following graph and divisor, with $q=v_2$:

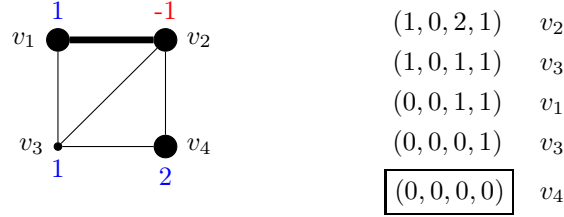
\begin{figure}[ht]
\centering
\begin{minipage}{0.3\textwidth}
\centering
\begin{tikzpicture}[node distance={15mm},main/.style = {draw,circle,fill, inner sep=3pt}]
\begin{scope}
\node[main,label = right:$v_2$,label = {[red]above:-1}] at (45:1cm) (1) {};
\node[main,label = left:$v_1$,label = {[blue]above:1}] at (135:1cm) (2) {};
\node[draw,circle,fill,inner sep=1pt,label = left:$v_3$,label = {[blue]below:1}] at (225:1cm) (3) {};
\node[main,label = right:$v_4$,label = {[blue]below:2}] at (315:1cm) (4) {};

\draw [line width = 0.9mm] (1) -- (2);
\draw (1) -- (3);
\draw (2) -- (3);
\draw (3) -- (4);
\draw (4) -- (1);
\end{scope}
\end{tikzpicture}
\end{minipage}%
\begin{minipage}{0.3\textwidth}
\begin{alignat*}{2}
    &\;(1,0,2,1) \quad && v_2 \\
    &\;(1,0,1,1) \quad && v_3 \\
    &\;(0,0,1,1) \quad && v_1 \\
    &\;(0,0,0,1) \quad && v_3\\
    &\boxed{(0,0,0,0)} \quad && v_4
\end{alignat*}
\end{minipage}
\caption[An unwinnable divisor and its algorithm run]{An unwinnable divisor and the firing scripts checked by the burning algorithm}
\label{fig:dhargraph}
\end{figure}

So we associate the word $v_2v_3v_1v_3v_4$. In order to reverse the association, first let $h:V \times V \rightarrow \mathbb{N}$ be defined by
\begin{equation}
    h^W_n(v,u) = \sum_{e \in E(u,v)} \dfrac{k^W_u(n)w(u)-k^W_{v}(n)w(v)}{w(e)}.
\label{eq:h}
\end{equation}
This quantity gives the net number of chips lost by $v$ to $u$ when firing the script $\sigma$ checked at the $n$th stage.

\begin{definition}
The divisor corresponding to a word $W$ on $\Gamma$ is 
\begin{equation}
    \mathbf{D}(W) = \sum_{v \in V(\Gamma)} \left(\min_{i \in \Ind_W(v)}\left\{f_i^W(v)\right\}-1\right)v,
\end{equation}
where $f_i^W(v) = \sum_{u \in V(\Gamma)} h_i^W(v,u)$. 
\end{definition}

This construction formalizes the following intuition: each time a vertex burns during the algorithm, it means that the potential firing given by that stage, which would remove $x$ chips from $v$, removes more chips than $v$ has, and so $D(v)<x$. The amount of chips removed in this potential firing is given by $f_i^W(v)$, and so $\mathbf{D}(W)(v)$ gives the maximum number of chips possible at $v$ for the algorithm to proceed in the order given by $W$.

\begin{theorem}\label{unwinnable}
Let $W$ be a word on $\Gamma$ such that $W(1) = q$ with $c(q) = 1$. Then $\mathbf{D}(W)$ is unwinnable. If $\mathbf{D}(W)$ is $q$-effective, it is also $q$-reduced.
\end{theorem}
\begin{proof}
Since $W(1) = q$, we have that $\mathbf{D}(W)(q) = -1$, and therefore $\mathbf{D}(W)$ is not an effective divisor. Now suppose that each $v \neq q$ has burned $k_{v}^W(i)$ times, and suppose $W(i)=u$. Then $u$ itself has burned $k_{u}^W(i)$ times, and therefore for each $v$, $u$ is sending $(c(u)-k_u^W(i))\frac{w(u)}{w(e)}$ chips along $e$ for each $e$ with roots $u$ and $v$, and is receiving a total of $(c(v)-k_v^W(i))\frac{w(v)}{w(e)}$ chips from $v$ along $e$. Thus the net number of chips being sent to $u$ along all $e \in E(u,v)$ is 
\begin{align*}
    \sum_{e \in E(u,v)} \left(\dfrac{(c(u)-k_u^W(i))w(u) - (c(v)-k_v^W(i))w(v)}{w(e)}\right) &= \sum_{e \in E(u,v)} \dfrac{k_v^W(i)w(v)-k_v^W(i)w(u)}{w(e)} = h_i^W(u,v).
\end{align*}
and the total number of chips being removed from $u$ at this step is $\sum_v h_i^W(u,v) =f_i^W(u)$. So if we run the burning algorithm on $\mathbf{D}(W)$, we have less chips at each stage than we can afford to send:
\[\mathbf{D}(W)(u) = \min_{i \in \Ind_W(u)}\left\{f_i^W(u)\right\} - 1 < f_i^W(u)\]
for all $i \in \Ind_W(u)$. Therefore $W(i) = u$ burns, and since this is true for each $i$ we have that the algorithm returns $\mathbf{0}$. Since $\mathbf{D}(W)(q)<0$, and there are no legal scripts on $V-\{q\}$, $\mathbf{D}(W)$ must be unwinnable. If $\mathbf{D}(W)$ is $q$-effective it clearly follows that $\mathbf{D}(W)$ is $q$-reduced.
\end{proof}

Note that not all divisors produced this way are $q$-effective (consider $W = v_1v_3v_3v_4v_2$ on the graph in Figure \ref{fig:dhargraph}).

\begin{proposition}
    Assume $c(q)=1$. Let $\mathcal{W}$ denote the set of words on $\Gamma$ with $W(1) = q$. Then each maximally unwinnable divisor $M \in \mathbf{D}(\mathcal{W})$.
\end{proposition}
\begin{proof}
    Recall that $\mathbf{D}(W)$ provides a divisor with the maximal number of chips at each vertex such that the algorithm proceeds in the order $W$. To verify this claim, suppose $D' > \mathbf{D}(W)$ runs in the order $W$. Let $v$ be such that $D'(v)>\mathbf{D}(W)(v)$. Then $D'(v) \geq \min_{i \in \Ind_W(v)}\{f_i^W(v)\}$. But then for some $i$ (in particular an $i$ at which the minimum is attained), $D'$ has at least as many chips at $v$ as are being sent away from $v$ in the $i$th-stage firing script, and therefore $v$ cannot burn at the $i$th stage, a contradiction. 
    
    Now to each maximally unwinnable divisor $M$ we can associate a word $W_M$ simply by running the Modified Burning Algorithm on $M$. If $\mathbf{D}(W_M) \neq M$ then we have an unwinnable divisor $\mathbf{D}(W_M)$ which dominates $M$, a contradiction. Thus $\mathbf{D}(W_M) = W$.
\end{proof}

In general, weighted graphs do not have a bijection between words with $W(1)=q$ and maximally unwinnable divisors. Considering every word on our graph $\Gamma$ in Figure 8, we see that there are exactly 12 words on $\Gamma$ starting with $v_1$. However, only 5 of these words produce maximally unwinnable divisors. It is easy enough to find the maximally unwinnable divisors by generating the entire list $\mathbf{D}(\mathcal{W})$; since every divisor in the list is $q$-reduced and the list contains all $q$-reduced maximally unwinnable divisors, a divisor in $\mathbf{D}(\mathcal{W})$ is maximally unwinnable if and only if it is not dominated by any other divisor in $\mathbf{D}(\mathcal{W})$. This is not a particularly enlightening criterion, however, as it gives us little insight as to what kind of words map to maximally unwinnable divisors.

\section{Concluding Remarks}

It would be optimistic to hope that this relationship between words and maximally unwinnable divisors could be leveraged to prove a generalization of the tropical Riemann-Roch Theorem to weighted graphs just as the relationship between acyclic orientations and maximally unwinnable divisors was used in the unweighted case. In particular, knowing that the canonical divisor on an unweighted graph can be defined by $K(v) = \val(v)-2$ or equivalently by $K = \mathbf{D}(\mathcal{O})+\mathbf{D}(\mathcal{O}_{rev})$ for some acyclic orientation $\mathcal{O}$, we may na\"{i}vely guess that a canonical divisor on a weighted graph may be defined identically using the weighted valency or by $\mathbf{D}(W)+\mathbf{D}{(W_{rev})}$. These can both be seen quickly to fail. On the graph in Figure \ref{fig:canonical}, which has weights 1 and 2, we should expect a canonical divisor $K = (1,3,1,2)$. However, if the Riemann-Roch Theorem in its standard formulation holds, we would expect a natural involution on the set of maximally unwinnable divisors given by $D \mapsto K-D$. This does not hold, as evidenced in Figure \ref{fig:canonical}. It is natural to consider the possibility that this can be remedied by a new notion of genus on a weighted graph which may take on fractional values.

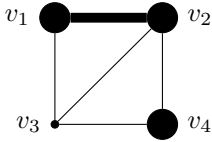
\begin{figure}[ht]
\centering
\begin{minipage}{0.3\textwidth}
\centering
\begin{tikzpicture}[node distance={15mm},main/.style = {draw,circle,fill, inner sep=4pt}]
\begin{scope}
\node[main,label = right:$v_2$] at (45:1cm) (1) {};
\node[main,label = left:$v_1$] at (135:1cm) (2) {};
\node[draw,circle,fill,inner sep=1pt,label = left:$v_3$] at (225:1cm) (3) {};
\node[main,label = right:$v_4$] at (315:1cm) (4) {};

\draw [line width = 1.3mm] (1) -- (2);
\draw (1) -- (3);
\draw (2) -- (3);
\draw (3) -- (4);
\draw (4) -- (1);
\end{scope}
\end{tikzpicture}
\end{minipage}%
\begin{minipage}{0.7\textwidth}
\begin{align*}
    [D_2]+[D_2] = [(-1,1,0,3)] + [(2,2,0,-1)] = [(1,3,0,2)] \\
    [D_1]+[D_5] = [(0,3,1,-1)] + [(1,0,-1,3)] = [(1,3,0,2)] \\
    [D_3]+[D_4] = [(-1,1,1,2)] + [(2,2,-1,0)] = [(1,3,0,2)] 
\end{align*}
\end{minipage}
\caption{While the maximally unwinnable divisors can be paired to make a potential canonical divisor, it is not given by $K(v) = \val(v)-2$.}
\label{fig:canonical}
\end{figure}

\end{document}